\documentclass[amsart,11pt]{article}
\usepackage{color,graphicx,latexsym,amssymb,amsmath,amsfonts,amsthm,float,enumerate}
\usepackage{authblk}

\usepackage{url}
\usepackage{hyperref}
\hypersetup{colorlinks}

\usepackage[font=scriptsize]{caption}
\usepackage{subcaption}
\usepackage{tikz}
\usetikzlibrary{spy} 
\usetikzlibrary{backgrounds}
\usetikzlibrary{decorations.markings}
\usepackage{tkz-graph}
\tikzstyle{vertex}=[circle, black, draw, inner sep=0.5pt, minimum size=1pt]
\newcommand{\vertex}{\node[vertex]}
\usetikzlibrary{arrows}

\pgfdeclarelayer{background layer}
\pgfsetlayers{background layer,main}
\definecolor{darkgray}{rgb}{0.25,0.25,0.25}
\definecolor{lightgray}{rgb}{0.75,0.75,0.75}

\newcommand{\N}{\mathbb{N}}

\newcommand{\Z}{\mathbb{Z}}
\newcommand{\R}{\mathbb{R}}
\newcommand{\PP}{\mathbb{P}}
\newcommand{\EE}{\mathbb{E}}
\newcommand\old[1]{}

\newcommand\floor[1]{{\left\lfloor #1 \right\rfloor}}
\def\eps{\epsilon}
\def\til{\widetilde}

\newtheorem{thm}{Theorem}[section]

\newtheorem{lemma}[thm]{Lemma}
\newtheorem{proposition}[thm]{Proposition}

\newtheorem{corollary}[thm]{Corollary}

\theoremstyle{definition}
\newtheorem*{definition}{Definition}
\newtheorem{remark}[thm]{Remark}

\newtheorem{conjecture}[thm]{Conjecture}

\title{The range of a rotor walk}
\author[1]{Laura Florescu\thanks{florescu@cims.nyu.edu}}
\author[2]{Lionel Levine\thanks{\url{http://www.math.cornell.edu/\~levine}. Partially supported by NSF grant DMS-1243606.}}
\author[3]{Yuval Peres\thanks{peres@microsoft.com}}
\affil[1]{New York University}
\affil[2]{Cornell University}
\affil[3]{Microsoft Research} 

\date{August 20, 2014}

\begin{document}

\maketitle

\abstract{
In a \emph{rotor walk} the exits from each vertex follow a prescribed periodic sequence.
On an infinite Eulerian graph embedded periodically in $\R^d$, we show that any simple rotor walk, regardless of rotor mechanism or initial rotor configuration, 
visits at least on the order of $t^{d/(d+1)}$ distinct sites in $t$ steps.  We prove a shape theorem for the rotor walk on the comb graph with i.i.d.\ uniform initial rotors, showing that the range is of order $t^{2/3}$ and the asymptotic shape of the range is a diamond. Using a connection to the mirror model and critical percolation, we show that rotor walk with i.i.d.\ uniform initial rotors is recurrent on two different directed graphs obtained by orienting the edges of the square grid, the Manhattan lattice and the $F$-lattice.  We end with a short discussion of the time it takes for rotor walk to cover a finite Eulerian graph.
}

\section{Introduction}

In a \emph{rotor walk} on a graph, the exits from each vertex follow a prescribed periodic sequence.
Such walks were first studied in \cite{WLB96} as a model of mobile agents exploring a territory, and in \cite{pddk} as a model of self-organized criticality.  Propp proposed rotor walk as a deterministic analogue of random walk, a perspective explored in \cite{CS,escaperates,holroyd-propp}. This paper is concerned with the following questions: 
How much territory does a rotor walk cover in $t$ steps? Conversely, how many steps does it take for a rotor walk to completely explore a given finite graph?

Let $G=(V,E)$ be a finite or infinite directed graph. For $v \in V$ let $E_v \subset E$ be the set of outbound edges from $v$, and let $\mathcal{C}_v$ be the set of all cyclic permutations of $E_v$.  A \emph{rotor configuration} on $G$ is a choice of outbound edge $\rho(v) \in E_v$ for each $v \in V$. A \emph{rotor mechanism} on $G$ is a choice of cyclic permutation $m(v) \in \mathcal{C}_v$ for each $v \in V$.  Given $\rho$ and $m$, the \emph{simple rotor walk} started at $X_0$ is a sequence of vertices $X_0, X_1, \ldots \in \mathbb{Z}^d$ and rotor configurations $\rho=\rho_0, \rho_1, \ldots$ such that for all integer times $t\geq 0$
	\[ \rho_{t+1}(v) = \begin{cases} m(v)(\rho_t(v)), & v=X_t \\ 
								\rho_t(v), & v \neq X_t \end{cases}
	\]
and
	\[ X_{t+1} = \rho_{t+1}(X_t)^+ \]
where $e^+$ denotes the target of the directed edge $e$.  In words, the rotor at $X_t$ ``rotates'' to point to a new neighbor of $X_{t}$ and then the walker steps to that neighbor.

In a simple rotor walk the sequence of exits from $v$ is periodic with period $\# E_v$. All rotor walks in this paper will be simple. (One can also study more general rotor walks in which the period is longer \cite{fs,holroyd-propp}.) We have chosen the retrospective rotor convention---each rotor at an already visited vertex indicates the direction of the most recent exit from that vertex---because it makes a few of our results such as Lemma~\ref{l.excursion2} easier to state.

\begin{figure}[H]
\captionsetup{width=0.8\textwidth}
\centering
\includegraphics[scale=.45]{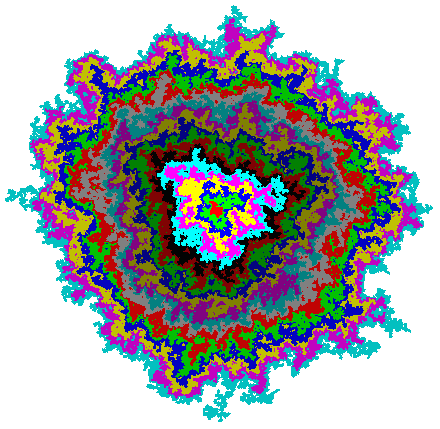}
\caption{The range of a clockwise uniform rotor walk on $\Z^2$ after $80$ returns to the origin: the mechanism $m$ cycles through the four neighbors in clockwise order (North, East, South, West), and the initial rotors $\rho(v)$ were oriented independently North, East, South or West, each with probability $1/4$. Colors indicate the first twenty excursion sets $A_1, \ldots, A_{20}$, defined in \textsection\ref{s.excursion}.}
\label{blobpic}
\end{figure}

 The \emph{range} of rotor walk at time $t$ is the set
	\[ R_t = \{X_1,\ldots,X_t \}. \]
We investigate the growth rate of the number of distinct sites visited, $\# R_t$.
A directed graph is called \emph{Eulerian} if each vertex has as many incoming as outgoing edges: indeg$(v)$=outdeg$(v)$ for all $v \in V$.  Any undirected graph can be made Eulerian by converting each undirected edge into a pair of oppositely oriented directed edges.

\begin{thm} 
\label{t.main}
For any Eulerian graph $G$ with a periodic embedding in $\R^d$, the number of distinct sites visited by a rotor walk in $t$ steps satisfies
	\[ \# R_t \geq c t^{d/(d+1)}. \]
for a constant $c>0$ depending only on $G$ (and not on $\rho$ or $m$).
\end{thm}

Priezzhev et al.\ \cite{pddk} and Povolotsky et al.\ \cite{pps} gave a heuristic argument that $\# R_t$ has order $t^{2/3}$ for the clockwise rotor walk on $\Z^2$ with uniform random initial rotors ($\rho(x) = \pm e_1, \pm e_2$ each with probability $1/4$, independently for each site $x$). Theorem~\ref{t.main} gives a lower bound of this order, and our proof is directly inspired by their argument. 

The upper bound promises to be more difficult because it depends on the initial rotor configuration $\rho$. Indeed, the next theorem shows that for certain $\rho$, the number of visited sites $\# R_t$ grows linearly in $t$.
Rotor walk is called \emph{recurrent} if $X_t=X_0$ for infinitely many $t$, and \emph{transient} otherwise. 

\begin{thm}
For any Eulerian graph $G$ and any mechanism $m$, if the initial rotor configuration $\rho$ has an infinite path of rotors directed toward $X_0$, then rotor walk is transient and \[ \# R_t \geq \frac{t}{\Delta}, \] 
where $\Delta$ is the maximal degree of a vertex in $G$.
\end{thm}

One can also ask about the shape of the random set $R_t$, pictured in Figure~\ref{blobpic}. Each pixel in this figure corresponds to a vertex of $\Z^2$, and $R_t$ is the set of all colored pixels (the different colors correspond to \emph{excursions} of the rotor walk, defined in \textsection\ref{s.excursion}); the mechanism $m$ is clockwise, and the initial rotors $\rho$ are i.i.d.\ uniform.
Although the set $R_t$ of Figure~\ref{blobpic} looks far from round, Kapri and Dhar have conjectured that for very large $t$ it becomes nearly a circular disk! From now on, by \textbf{uniform rotor walk} we will always mean that the initial rotors $\{\rho(v)\}_{v\in V}$ are independent and uniformly distributed on $E_v$.

\begin{conjecture}[Kapri-Dhar \cite{kapri-dhar}]  
\label{c.KD}
The set of sites $R_t$ visited by the clockwise uniform rotor walk in $\Z^2$ is asymptotically a disk: There exists a constant $c$ such that for any $\eps>0$, \[ P \{ D_{(c-\eps)t^{1/3}} \subset R_t \subset D_{(c+\eps)t^{1/3}} \} \to 1 \]
as $t \to \infty$, where $D_r = \{(x,y) \in \Z^2 \,:\, x^2 + y^2 < r^2\}$.
\end{conjecture}

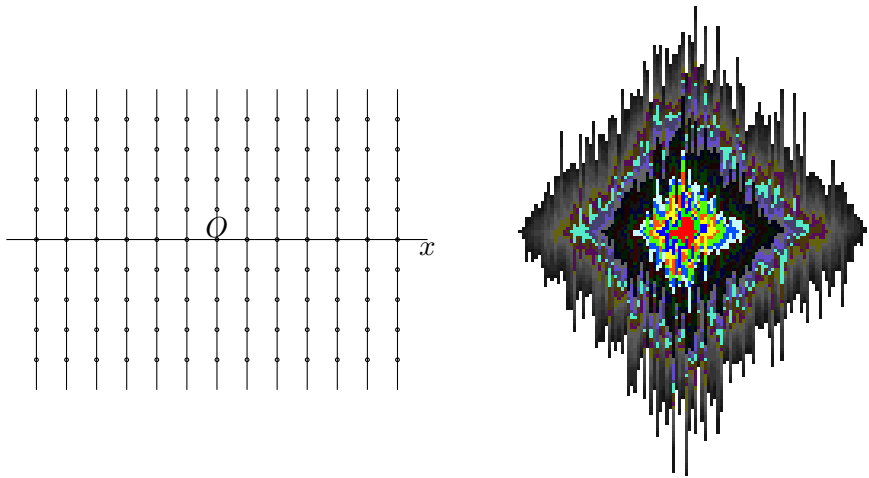
\begin{figure}[h!]
\captionsetup{width=0.8\textwidth}
\begin{minipage}{7in}
  \centering
$\vcenter{\hbox{
 \begin{tikzpicture}[scale=0.4] 
 \draw[color=black] (-6,-5) -- (-6,5); \draw[color=black] (-5,-5) -- (-5,5);
  \draw[color=black] (-4,-5) -- (-4,5); \draw[color=black] (-3,-5) -- (-3,5);
   \draw[color=black] (-2,-5) -- (-2,5); \draw[color=black] (-1,-5) -- (-1,5);
    \draw[color=black] (0,-5) -- (0,5); \draw[color=black] (1,-5) -- (1,5);
     \draw[color=black] (2,-5) -- (2,5); \draw[color=black] (3,-5) -- (3,5);
      \draw[color=black] (4,-5) -- (4,5); \draw[color=black] (5,-5) -- (5,5);
       \draw[color=black] (6,-5) -- (6,5); \draw[color=black] (-7,0) -- (7,0);
       \node[color=black] at (0,0.4) {$O$};
       \node[color=black] at (7,-0.3) {$x$}; 
               \vertex at (-1,0) {};
               \vertex at (-2,0) {};
               \vertex at (-3,0) {};
               \vertex at (-4,0) {};
               \vertex at (-5,0) {};
               \vertex at (-6,0) {};
               \vertex at (0,0) {};
	     \vertex at (1,0) {};
               \vertex at (2,0) {};
               \vertex at (3,0) {};
               \vertex at (4,0) {};
               \vertex at (5,0) {};
               \vertex at (6,0) {};

         \vertex at (-1,-1) {};
               \vertex at (-2,-1) {};
               \vertex at (-3,-1) {};
               \vertex at (-4,-1) {};
               \vertex at (-5,-1) {};
               \vertex at (-6,-1) {};
               \vertex at (0,-1) {};
	     \vertex at (1,-1) {};
               \vertex at (2,-1) {};
               \vertex at (3,-1) {};
               \vertex at (4,-1) {};
               \vertex at (5,-1) {};
               \vertex at (6,-1) {};
                        
                                 \vertex at (-2,-2) {};
               \vertex at (-1,-2) {};
               \vertex at (-3,-2) {};
               \vertex at (-4,-2) {};
               \vertex at (-5,-2) {};
               \vertex at (-6,-2) {};
               \vertex at (0,-2) {};
	     \vertex at (1,-2) {};
               \vertex at (2,-2) {};
               \vertex at (3,-2) {};
               \vertex at (4,-2) {};
               \vertex at (5,-2) {};
               \vertex at (6,-2) {};
   
       \vertex at (-1,-3) {};
               \vertex at (-2,-3) {};
               \vertex at (-3,-3) {};
               \vertex at (-4,-3) {};
               \vertex at (-5,-3) {};
               \vertex at (-6,-3) {};
               \vertex at (0,-3) {};
	     \vertex at (1,-3) {};
               \vertex at (2,-3) {};
               \vertex at (3,-3) {};
               \vertex at (4,-3) {};
               \vertex at (5,-3) {};
               \vertex at (6,-3) {};
                          \vertex at (-1,-4) {};
               \vertex at (-2,-4) {};
               \vertex at (-3,-4) {};
               \vertex at (-4,-4) {};
               \vertex at (-5,-4) {};
               \vertex at (-6,-4) {};
               \vertex at (0,-4) {};
	     \vertex at (1,-4) {};
               \vertex at (2,-4) {};
               \vertex at (3,-4) {};
               \vertex at (4,-4) {};
               \vertex at (5,-4) {};
               \vertex at (6,-4) {};
                                    \vertex at (-1,4) {};
               \vertex at (-2,4) {};
               \vertex at (-3,4) {};
               \vertex at (-4,4) {};
               \vertex at (-5,4) {};
               \vertex at (-6,4) {};
               \vertex at (0,4) {};
	     \vertex at (1,4) {};
               \vertex at (2,4) {};
               \vertex at (3,4) {};
               \vertex at (4,4) {};
               \vertex at (5,4) {};
               \vertex at (6,4) {};
               
                        \vertex at (-1,3) {};
               \vertex at (-2,3) {};
               \vertex at (-3,3) {};
               \vertex at (-4,3) {};
               \vertex at (-5,3) {};
               \vertex at (-6,3) {};
               \vertex at (0,3) {};
	     \vertex at (1,3) {};
               \vertex at (2,3) {};
               \vertex at (3,3) {};
               \vertex at (4,3) {};
               \vertex at (5,3) {};
               \vertex at (6,3) {};
               
                        \vertex at (-1,2) {};
               \vertex at (-2,2) {};
               \vertex at (-3,2) {};
               \vertex at (-4,2) {};
               \vertex at (-5,2) {};
               \vertex at (-6,2) {};
               \vertex at (0,2) {};
	     \vertex at (1,2) {};
               \vertex at (2,2) {};
               \vertex at (3,2) {};
               \vertex at (4,2) {};
               \vertex at (5,2) {};
               \vertex at (6,2) {};
                      \vertex at (-1,1) {};
               \vertex at (-2,1) {};
               \vertex at (-3,1) {};
               \vertex at (-4,1) {};
               \vertex at (-5,1) {};
               \vertex at (-6,1) {};
               \vertex at (0,1) {};
	     \vertex at (1,1) {};
               \vertex at (2,1) {};
               \vertex at (3,1) {};
               \vertex at (4,1) {};
               \vertex at (5,1) {};
               \vertex at (6,1) {};  
            \end{tikzpicture}
            }}$
                \hspace*{.2in}
            $\vcenter{\hbox{\includegraphics[height=2.5in]{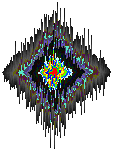}}}$
\end{minipage}
\caption{A piece of the comb graph (left) and the set of sites visited by a uniform rotor walk on the comb graph in 10000 steps.
\label{f.comb}}
\end{figure}

We are a long way from proving anything like Conjecture \ref{c.KD}, but we can show that an analogous shape theorem holds on a much simpler graph, the two dimensional comb (Figure~\ref{f.comb}).

\begin{thm} 
\label{t.combintro}
For uniform rotor walk on the comb graph, $\# R_t$ has order $t^{2/3}$ and the asymptotoic shape of $R_t$ is a diamond.
\end{thm}

For the precise statement, see \textsection\ref{s.comb}. This result contrasts with random walk on the comb, for which the expected number of sites visited is only on the order of $t^{1/2} \log t$ as shown by Pach and Tardos \cite{pach-tardos}.
Thus the uniform rotor walk explores the comb more efficiently than random walk. (On the other hand, it is conjectured to explore $\Z^2$ \emph{less} efficiently than random walk!)

The main difficulty in proving upper bounds for $\# R_t$ lies in showing that the uniform rotor walk is recurrent. This seems to be a difficult problem in $\Z^2$, but we can show it for two different directed graphs obtained by orienting the edges of $\Z^2$: the Manhattan lattice and the $F$-lattice, pictured in Figure~\ref{f.directed}.

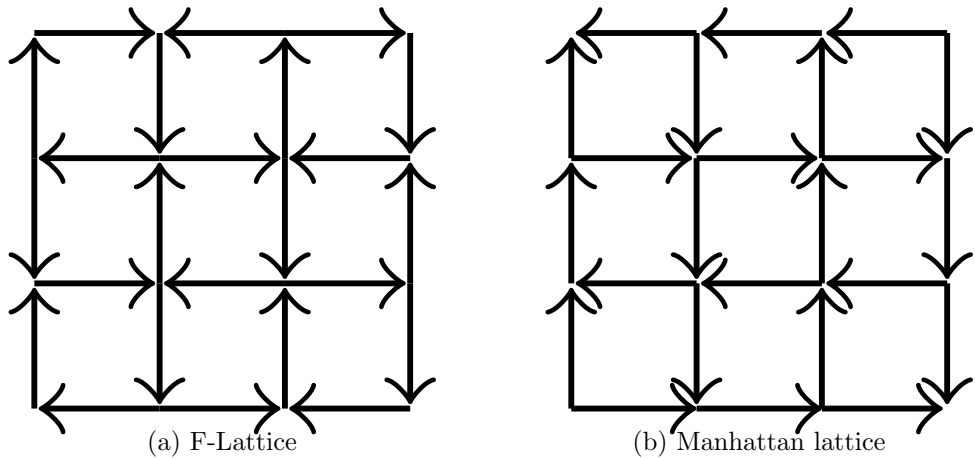
\begin{figure}[h!]
\centering
    \begin{subfigure}[b]{0.38\textwidth}
        \centering
        \resizebox{\linewidth}{!}{
            \begin{tikzpicture}[->,scale=0.3]
 \draw[color=black, shorten >=0.3pt] (-2,-2) -- (-2,-1);\draw[color=black, shorten >=0.3pt] (-1,-2) -- (0,-2);\draw[color=black, shorten >=0.3pt] (-1,-2) -- (-2,-2);\draw[color=black, shorten >=0.3pt] (0,-2) -- (0,-1);\draw[color=black, shorten >=0.3pt] (1,-2) -- (0,-2);\draw[color=black, shorten >=0.3pt] (-2,-1) -- (-1,-1);\draw[color=black, shorten >=0.3pt] (-1,-1) -- (-1,0);\draw[color=black, shorten >=0.3pt] (-1,-1) -- (-1,-2);\draw[color=black, shorten >=0.3pt] (0,-1) -- (1,-1);\draw[color=black, shorten >=0.3pt] (0,-1) -- (-1,-1);\draw[color=black, shorten >=0.3pt] (1,-1) -- (1,0);\draw[color=black, shorten >=0.3pt] (1,-1) -- (1,-2);\draw[color=black, shorten >=0.3pt] (-2,0) -- (-2,1);\draw[color=black, shorten >=0.3pt] (-2,0) -- (-2,-1);\draw[color=black, shorten >=0.3pt] (-1,0) -- (0,0);\draw[color=black, shorten >=0.3pt] (-1,0) -- (-2,0);\draw[color=black, shorten >=0.3pt] (0,0) -- (0,1);\draw[color=black, shorten >=0.3pt] (0,0) -- (0,-1);\draw[color=black, shorten >=0.3pt] (1,0) -- (0,0);\draw[color=black, shorten >=0.3pt] (-2,1) -- (-1,1);\draw[color=black, shorten >=0.3pt] (-1,1) -- (-1,0);\draw[color=black, shorten >=0.3pt] (0,1) -- (1,1);\draw[color=black, shorten >=0.3pt] (0,1) -- (-1,1);\draw[color=black, shorten >=0.3pt] (1,1) -- (1,0);
            \end{tikzpicture}
        }
        \caption{F-Lattice}
        \label{fig:subfig8}
    \end{subfigure}
   $\quad$
    \begin{subfigure}[b]{0.38\textwidth}
    \centering
        \resizebox{\linewidth}{!}{
            \begin{tikzpicture}[->,scale=0.3]
            \draw[color=black] (-2,-2) -- (-1,-2);
\draw[color=black,= shorten >=0.3pt] (-1,-2) -- (0,-2);
\draw[color=black, shorten >=0.3pt] (0,-2) -- (1,-2);\draw[color=black, shorten >=0.3pt] (-1,-1) -- (-2,-1);\draw[color=black, shorten >=0.3pt] (0,-1) -- (-1,-1);\draw[color=black, shorten >=0.3pt] (1,-1) -- (0,-1);\draw[color=black, shorten >=0.3pt] (-2,0) -- (-1,0);\draw[color=black, shorten >=0.3pt] (-1,0) -- (0,0);\draw[color=black, shorten >=0.3pt] (0,0) -- (1,0);\draw[color=black, shorten >=0.3pt] (-1,1) -- (-2,1);\draw[color=black, shorten >=0.3pt] (0,1) -- (-1,1);\draw[color=black, shorten >=0.3pt] (1,1) -- (0,1);\draw[color=black, shorten >=0.3pt] (-2,-2) -- (-2,-1);\draw[color=black, shorten >=0.3pt] (-2,-1) -- (-2,0);\draw[color=black, shorten >=0.3pt] (-2,0) -- (-2,1);
\draw[color=black, shorten >=0.3pt] (-1,-1) -- (-1,-2);\draw[color=black, shorten >=0.3pt] (-1,0) -- (-1,-1);\draw[color=black, shorten >=0.3pt] (-1,1) -- (-1,0);\draw[color=black, shorten >=0.3pt] (0,-2) -- (0,-1);\draw[color=black, shorten >=0.3pt] (0,-1) -- (0,0);\draw[color=black, shorten >=0.3pt] (0,0) -- (0,1);
\draw[color=black, shorten >=0.3pt] (1,-1) -- (1,-2);\draw[color=black, shorten >=0.3pt] (1,0) -- (1,-1);\draw[color=black, shorten >=0.3pt] (1,1) -- (1,0);

            \end{tikzpicture}
        }
        \caption{Manhattan lattice}   
        \label{fig:subfig9}
    \end{subfigure}
    \caption{Two different periodic orientations of the square grid with indegree and outdegree $2$.}
    \label{f.directed}
\end{figure}

\begin{thm}
\label{t.directed}
Uniform rotor walk is recurrent on both the $F$-lattice and the Manhattan lattice.
\end{thm}

The proof uses a connection to the mirror model and critical bond percolation on $\Z^2$; see \textsection\ref{s.pinball}.

Theorems \ref{t.main}-\ref{t.directed} bound the rate at which rotor walk explores various infinite graphs. In \textsection\ref{s.cover} we bound the time it takes a rotor walk to completely explore a given finite graph.

\subsection*{Related work}

By comparing to a branching process, Angel and Holroyd \cite{angel-holroyd-trees} showed that uniform rotor walk on the infinite $b$-ary tree is transient for $b \geq 3$ and recurrent for $b=2$. In the latter case the corresponding branching process is critical, and the distance traveled by rotor walk before returning $n$ times to the root is doubly exponential in $n$. They also studied rotor walk on a singly infinite comb with the ``most transient'' initial rotor configuration $\rho$. They showed that if $n$ particles start at the origin then order $\sqrt{n}$ of them escape to infinity (more generally, order $n^{1-2^{1-d}}$ for a $d$-dimensional analogue of the comb). 

In \emph{rotor aggregation}, each of $n$ particles starting at the origin performs rotor walk until reaching an unoccupied site, which it then occupies. For rotor aggregation in $\Z^d$, the asymptotic shape of the set of occupied sites is a Euclidean ball \cite{levine-peres}.  For the layered square lattice ($\Z^2$ with an outward bias along the $x$- and $y$-axes) the asymptotic shape becomes a diamond \cite{KL}.  Huss and Sava \cite{HScomb} studied rotor aggregation on the $2$-dimensional comb with the ``most recurrent'' initial rotor configuration. They showed that at certain times the boundary of the set of occupied sites is composed of four segments of exact parabolas. It is interesting to compare their result with Theorem~\ref{t.combintro}: The asymptotic shape, and even the scaling required (elliptic for rotor walk, parabolic for rotor aggregation), is different.

\old{
\begin{table}[h]
\begin{center}
        \begin{tabular}{| l l l | l |  p{2cm} |}

    \hline
    Lattice & $ \# R_t$ (experimental) & $\# R_t$ (rigourous) \\ \hline
    $\mathbb{Z}^2$ (N-E-S-W) & $6.0000\, t^{2/3}$ & $\Omega(t^{2/3})$ \\ \hline
    $\mathbb{Z}^2$ (N-S-E-W) & $\Theta(t)$ &  $\Omega(t^{2/3})$ \\ \hline
    triangular lattice & $11.12\, t^{2/3}$ & $\Omega(t^{2/3})$ \\ \hline
        Manhattan lattice &  $16.5\, t^{2/3}$ & $\Omega(t^{2/3})$ \\ \hline
                        $F$-lattice &  $\Theta(t)$ & $\Omega(t^{2/3})$ \\ \hline
    \end{tabular}
    \centering
 \caption{Number of distinct sites visited by the uniform rotor walk in $t$ steps, on various two-dimensional lattices.}
\end{center}
    \end{table}
}

\section{Excursions}
\label{s.excursion}

Let $G=(V,E)$ be a connected Eulerian graph. In this section $G$ can be either finite or infinite, and the rotor mechanism $m$ can be arbitrary. The main idea of the proof of Theorem~\ref{t.main} is to decompose rotor walk on $G$ into a sequence of excursions.  This idea was also used in \cite{angel-holroyd-rec} to construct recurrent rotor configurations on $\Z^d$ for all $d$, and in \cite{lockin,mazes,patrol} to bound the cover time of rotor walk on a finite graph (about which we say more in \textsection\ref{s.cover}).

\begin{definition} 
 Fix a vertex $o \in V$.
An \emph{excursion} from $o$ is a rotor walk started at $o$ and run until it returns to $o$ exactly $\deg(o)$ times. 
\end{definition}

More formally, let $(X_t)_{t \geq 0}$ be a rotor walk started at $X_0 = o$. For $t\geq 0$ let
	\[ u_t(x) = \# \{0 \leq s<t \,:\, X_s = x\} \]
and let $u_\infty(x) \in \N \cup \{\infty\}$ be the increasing limit of $u_t(x)$. For $n \geq 0$ let
	\[ T(n) = \inf \{ t \geq 0 \,:\, u_t(o) \geq n \deg(o) \} \in \N \cup \{\infty\}\]
be the time taken for the rotor walk to complete $n$ excursions from $o$.  For all $n \geq 1$ such that $T(n-1)<\infty$, define
	\[ e_{n} = u_{T(n)} - u_{T(n-1)}, \qquad n \geq 1. \]
Our first lemma says that each $x \in V$ is visited at most $\deg(x)$ times per excursion. The assumption that $G$ is Eulerian is crucial here.

\begin{lemma} \cite[Lemma~8]{angel-holroyd-rec}; \cite[\textsection4.2]{mazes}
\label{l.excursion1}
For any initial rotor configuration $\rho$, 
	\[ e_1(x) \leq \deg(x)  \qquad \forall x \in V. \]
\end{lemma}

\begin{proof}
If the rotor walk never traverses the same directed edge twice, then $T(1)=\infty$ and $u_\infty \leq \deg$, so we are done. Otherwise, consider the smallest $t$ such that $(X_s,X_{s+1}) = (X_t,X_{t+1})$ for some $s<t$. Rotor walk reuses an outgoing edge from $X_t$ only after it has used all of the outgoing edges from $X_t$.  Therefore, at time $t$ the vertex $X_t$ has been visited $\deg(X_t)+1$ times, but each incoming edge to $X_t$ has been traversed at most once. Since $G$ is Eulerian it follows that $X_t=o$ and $t = T(1)$.
\end{proof}

\begin{lemma}
\label{l.excursion2}
If $T(1)<\infty$ and there is a directed path of \underline{initial} rotors from $x$ to $o$, then 
	\[ e_1(x) = \deg(x). \]
\end{lemma}

\begin{proof}
Let $y$ be the first vertex on the path of initial rotors from $x$ to $o$. By induction on the length of this path, $y$ is visited exactly $\deg(y)$ times in an excursion from $o$. Each incoming edge to $y$ is traversed at most once by Lemma~\ref{l.excursion1}, so in fact each incoming edge to $y$ is traversed exactly once.  In particular, the edge $(x,y)$ is traversed. Since $\rho(x)=(x,y)$, the edge $(x,y)$ is the last one traversed out of $x$, so $x$ must be visited at least $\deg(x)$ times.
\end{proof}

If $G$ is finite, then $T(n)<\infty$ for all $n$ by Lemma~\ref{l.excursion1}. If $G$ is infinite, then depending on the rotor mechanism $m$ and initial rotor configuration $\rho$, rotor walk may or may not complete an excursion from $o$. In particular, Lemma~\ref{l.excursion2} implies the following.

\begin{corollary}
\label{c.incomplete}
If $\rho$ has an infinite path directed toward $o$, then $T(1)=\infty$.
\end{corollary}
	
Now let
	\[ A_n = \{ x \in V \,:\, e_n(x)>0 \} \]
be the set of sites visited during the $n$th excursion.   We also set $e_0 = \delta_o$ and $A_0 = \{o\}$. 
For a subset $A \subset V$, define
	\[ \partial A = \{y \in V \,:\, (x,y) \in E \; \text{for some } x \in A\}. \]
	
\begin{lemma}
\label{l.excursion3}
If $T(n+1)<\infty$, then
\begin{enumerate}[\em (i)]
\item $e_{n+1}(x) \leq \deg(x)$ for all $x \in V$.
\item $e_{n+1}(x) = \deg(x)$ for all $x \in A_n$.
\item $A_{n+1} \supseteq A_n \cup \partial A_n$.
\end{enumerate}
\end{lemma}

\begin{proof}
Part (i) is immediate from Lemma~\ref{l.excursion1}.

Part (ii) follows from Lemma~\ref{l.excursion2} and the observation that in the rotor configuration $\rho_{T(n)}$, the rotor at each $x \in A_n$ points along the edge traversed most recently from $x$, so for each $x \in A_n$ there is a directed path of rotors in $\rho_{T(n)}$ leading to $X_{T(n)} = o$. 

Part (iii) follows from (ii): the $(n+1)$st excursion traverses each outgoing edge from each $x\in A_n$, so in particular it visits each vertex in $A_n \cup \partial A_n$.
%
\end{proof}

For $x \in V$ and $r \in \N$ denote by $B(x,r)$ the set of vertices reachable from $x$ by a directed path of length $\leq r$. 
Inducting on $n$ using Lemma~\ref{l.excursion3}(ii), we obtain the following.

\begin{corollary}
\label{c.balls}
If $T(n)<\infty$, then $B(o,n) \subseteq A_n$.
\end{corollary}

Rotor walk is called \emph{recurrent} if $T(n)<\infty$ for all $n$.  
Consider the rotor configuration $\rho_{T(n)}$ at the end of the $n$th excursion. By Lemma~\ref{l.excursion3}, each vertex in $x \in A_n$ is visited exactly $\deg(x)$ times during the $N$th excursion for each $N \geq n+1$, so we obtain the following.

\begin{corollary}
For a recurrent rotor walk, 
$\rho_{T(N)}(x) = \rho_{T(n)}(x)$ for all $x \in A_n$ and all $N \geq n$.
\end{corollary}

The following proposition is a kind of converse to Lemma~\ref{l.excursion3} in the case of undirected graphs.

\begin{proposition}
\label{t.AH}
\cite[Lemma 3]{lockin}; \cite[Prop.\ 11]{angel-holroyd-rec}
Let $G=(V,E)$ be an undirected graph. For sequence of connected sets $S_1, S_2, \ldots \subset V$ such that $S_{n+1} \supseteq S_{n} \cup \partial S_n$ for all $n \geq 1$, and any vertex $o \in S_1$, there exists a rotor mechanism $m$ and initial rotors $\rho$ such that 
the $n$th excursion for rotor walk started at $o$ traverses each edge incident to $S_n$ exactly once in each direction, and no other edges. \end{proposition}

\section{Lower bound on the range}
\label{s.lowerbound}

In this section $G=(V,E)$ is an infinite connected Eulerian graph.
For $x \in V$ and $r \in \N$ denote by $B(x,r)$ the set of vertices reachable from $x$ by a directed path of length $\leq r$. Fix an origin $o \in V$ and let 
$v(r)$ be the number of directed edges incident to $B(o,r)$. 
Let $W(r) = \sum_{n=0}^{r-1} v(n)$.  Write $W^{-1}(t) = \min \{r\in \N \,:\, W(r) > t\}$. 

Fix a rotor mechanism $m$ and an initial rotor configuration $\rho$ on $G$.
For $x \in V$ let $u_t(x)$ be the number of times $x$ is visited by a rotor walk started at $o$ and run for $t$ steps.
The \emph{range} of rotor walk is the set $R_t = \{x \in V \,:\, u_t(x)>0 \}$.

\begin{thm}
\label{t.lowerbound}
For any rotor mechanism $m$, any initial rotor configuration $\rho$ on $G$, and any time $t \geq 0$, the following bounds hold.
\begin{enumerate}[\em (i)]
\item $u_t(o) < \deg(o) W^{-1}(t)$.
\item $u_t(x) \leq u_t(o) + \deg(x)$ for all $x \in V$.
\item Let $\Delta_t = \max_{x \in B(o,t)} \deg(x)$. Then
	\begin{equation} \label{e.rangelowerbound} \# R_t \geq \frac{t}{\deg(o) W^{-1}(t)+\Delta_t - 1} \end{equation}
\end{enumerate}
\end{thm}

Before proving this theorem we discuss a few examples. If $G= \Z^2$ then $B(o,r)$ is a diamond 
and $W(r) \sim \frac83 r^3$, which gives $\# R_t \geq c t^{2/3}$ with $c = (8/3)^{1/3} \approx 1.387$.  
\old{
Experiments show $\#R_t$ is actually about $6t^{2/3}$ for clockwise $m$ and i.i.d.\ uniform $\rho$. On the other hand, there exist $m$ and $\rho$ such that $\# R_t \sim (8/3)^{1/3} t^{2/3}$; see the remark following the proof.
}
More generally, if $G$ is any graph with a periodic embedding in $\R^d$, then $W(r) = \Omega( r^{d+1} )$, so by part (iii) the range of any rotor walk on $G$ is at least $\# R_t = \Omega( t^{d/(d+1)} )$, which shows that Theorem~\ref{t.lowerbound} implies Theorem~\ref{t.main}.
If $G$ has exponential volume growth we get $\# R_t = \Omega( t/\log t)$.

\begin{proof}[Proof of Theorem~\ref{t.lowerbound}]

By Lemma~\ref{l.excursion3} and Corollary~\ref{c.balls}, the $n$th excursion from $o$ traverses each directed edge incident to $B(o,n-1)$, so the total length of the first $r$ excursions is at least $W(r)$.  Therefore if $t < W(r)$ then the rotor walk has not yet completed its $r$th excursion at time $t$, so $u_t(o) < r \deg(o)$. Taking $r=W^{-1}(t)$ yields part (i).

Part (ii) is immediate from Lemma~\ref{l.excursion1}.

Part (iii) follows from the fact that $t = \sum_{x \in B(o,t)} u_t(x)$: By parts (i) and (ii), each term is at most $\deg(o) W^{-1}(t) - 1 +\Delta_t$, so there are at least $t/(\deg(o) W^{-1}(t)-1+\Delta_t)$ nonzero terms.
\end{proof}

\begin{remark}
Theorem~\ref{t.AH} shows that if $G$ is undirected, then \eqref{e.rangelowerbound} is the best possible lower bound on $\# R_t$ that does not depend on $m$ or $\rho$.  For example, taking $S_n = B(o,n)$ in $\Z^2$ yields a rotor walk with $\# R_t \sim (8/3)^{1/3} t^{2/3}$; the rotor mechanism is clockwise and the initial rotors are shown in Figure~\ref{f.diamondZ2}.
More generally, by taking $S_n$ to be a suitably growing sequence of sets, one can obtain any growth rate for $\# R_t$ intermediate between $t/W^{-1}(t)$ and~$t$. 
\end{remark}

\begin{center}

\begin{figure}[h!]
\centering
\begin{tikzpicture}
    [scale=1,only marks,reddot/.style={fill=red,circle,inner sep=0pt, minimum width=2pt},bluedot/.style={fill=blue,circle, inner sep=0pt, minimum width=2pt},
     spy using outlines={rectangle,lens={scale=3}, size=6cm, connect spies},
]

\draw[color=black] (-0.5,2) -- (0.5,2);
\draw[color=black] (-0.5,2) -- (-0.5,1);
\draw[color=black] (-0.5,1) -- (-1.5,1);
\draw[color=black] (-1.5,1) -- (-1.5,0);
\draw[color=black] (-1.5,0) -- (-2.5,0);
\draw[color=black] (-2.5,0) -- (-2.5,-1);
\draw[color=black] (-2.5,-1) -- (-1.5,-1);
\draw[color=black] (-1.5,-1) -- (-1.5,-2);
\draw[color=black] (-1.5,-2) -- (-0.5,-2);
\draw[color=black] (-0.5,-2) -- (-0.5,-3);
\draw[color=black] (-0.5,-3) -- (0.5,-3);
\draw[color=black] (0.5,-3) -- (0.5,-2);
\draw[color=black] (0.5,-2) -- (1.5,-2);
\draw[color=black] (1.5,-2) -- (1.5,-1);
\draw[color=black] (1.5,-1) -- (2.5,-1);
\draw[color=black] (2.5,-1) -- (2.5,0);
\draw[color=black] (2.5,0) -- (1.5,0);
\draw[color=black] (1.5,0) -- (1.5,1);
\draw[color = black] (1.5,1) -- (0.5,1);
\draw[color=black] (0.5,1) -- (0.5,2);

\draw[color=black][->] (-0.1,2.5) -- (0.4,2.5);
\draw[color=black][->] (-0.1,1.5) -- (0.4,1.5);
\draw[color=black][->] (1,1.5) -- (1.4,1.5);
\draw[color=black][->] (-0.1,0.5) -- (0.4,0.5);
\draw[color=black][->] (1,0.5) -- (1.4,0.5);

\draw[color=black][->] (1.2,-0.4) -- (1.2,-0.8);
\draw[color=black][->] (1.2,-1.3) -- (1.2,-1.7);
\draw[color=black][->] (1.2,-2.2) -- (1.2,-2.6);

\draw[color=black][->] (-1,-2.5) -- (-1.4,-2.5);
\draw[color=black][->] (-1,-1.5) -- (-1.4,-1.5);
\draw[color=black][->] (-1,-0.7) -- (-1,-0.3);
\draw[color=black][->] (-1,0.3) -- (-1,0.7);
\draw[color=black][->] (-1,1.3) -- (-1,1.7);

\draw[color=black][->] (-2,0.3) -- (-2,0.7);
\draw[color=black][->] (-2,-0.7) -- (-2,-0.3);
\draw[color=black][->] (-2,-1.5) -- (-2.4,-1.5);

\draw[color=black][->] (2,0.3) -- (2.4,0.3);
\draw[color=black][->] (2,-0.3) -- (2,-0.7);
\draw[color=black][->] (2,-1.3) -- (2,-1.7);

\draw[color=black][->] (3,-0.3) -- (3,-0.7);
\draw[color=black][->] (-3,-0.7) -- (-3,-0.3);

\draw[color=black][->] (0.4,-1.5) -- (-0.1,-1.5);
\draw[color=black][->] (0.4,-2.5) -- (-0.1,-2.5);
\draw[color=black][->] (0.4,-3.5) -- (-0.1,-3.5);

\end{tikzpicture}
\caption{\label{f.diamondZ2} Minimal range rotor configuration for $\mathbb{Z}^2$: the excursion sets are diamonds.}
\end{figure}
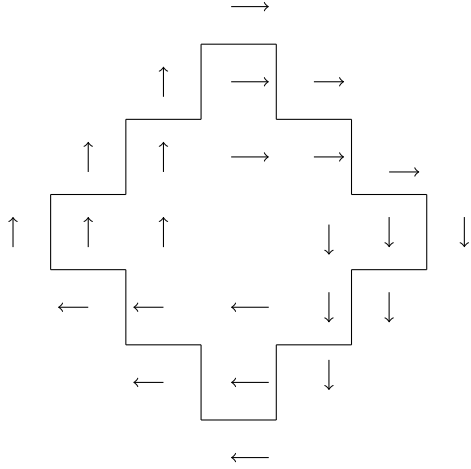

\end{center}

Part (i) of the next theorem gives a sufficient condition for rotor walk to be transient. Part (iii) shows that on a graph of bounded degree, the number of visited sites $\#R_t$ of a transient rotor walk grows linearly in $t$.

\begin{thm}
On any Eulerian graph, the following hold.
\begin{enumerate}[\em (i)]
\item If $\rho$ has an infinite path of initial rotors directed toward the origin $o$, then $u_t(o) < \deg(o)$ for all $t \geq 1$.
\item If $u_t(o) < \deg(o)$, then $\# R_t \geq t/\Delta_t$ where $\Delta_t = \max_{x \in B(o,t)} \deg(x)$.
\item If rotor walk is transient, then there is a constant $C=C(m,\rho)$ such that
	\[ \# R_t \geq \frac{t}{\Delta_t} - C \]
for all $t \geq 1$.
\end{enumerate}
\end{thm}

\begin{proof}
(i) By Corollary~\ref{c.incomplete}, if $\rho$ has an infinite path directed toward $o$, then rotor walk never completes its first excursion from $o$.

(ii) If rotor walk does not complete its first excursion, then it visits each vertex $x$ at most $\deg(x)$ times by Lemma~\ref{l.excursion1}, so it must visit at least $t/\Delta_t$ distinct vertices.

(iii) If rotor walk is transient, then for some $n$ it does not complete its $n$th excursion, so this follows from part (b) taking $C$ to be the total length of the first $n-1$ excursions.
\end{proof}

\old{ 
In particular, for $\mathbb{Z}^2$ the exact constant can be found as follows. First, we know that at time $t$, the set of sites visited is $\#R_t \ge (2n)^2$, where $n$ is the number of the current excursion. Summing the number of iterations until the $n$th excursion we have that $t  = \sum_{k=1}^{n} (2k)^2 = O( \frac{4n^3}{3})$. From this we get that $t^{2/3} = \left(\frac{4}{3}\right)^{2/3} n^2$, which implies that $\frac{\#R_t}{t^{2/3}} = 4^{1/3} 3^{2/3}$. 
\\
}

\old{
\begin{figure}
\centering
\begin{subfigure}{0.5\textwidth}
  \centering
  \includegraphics[width=1\linewidth]{../code/triangle2.png}
  \label{fig:sub1}
\end{subfigure}%
\begin{subfigure}{0.5\textwidth}
  \centering
  \includegraphics[width=1\linewidth]{../code/triangle5.png}
  \label{fig:sub2}
\end{subfigure}
\caption{Set of sites visited by rotor walk on the triangular lattice after the $2$nd and $15$th excursion respectively.}
\label{trianglepic}
\end{figure}
}

\section{Uniform rotor walk on the comb}
\label{s.comb}

The 2-dimensional \emph{comb} is the subgraph of the square lattice $\mathbb{Z}^2$ obtained by removing all of its horizontal edges except for those on the $x$-axis (Figure~\ref{f.comb}). Vertices on the $x$-axis have degree 4, and all other vertices have degree $2$.  

Recall that the \textbf{uniform rotor walk} starts with independent random initial rotors $\rho(v)$ with the uniform distribution on outgoing edges from $v$. The following result shows that the range of the uniform rotor walk on the comb is close to the diamond 
	\[ D_n := \{(x,y) \in \Z^2 \,:\, |x| + |y| < n \}. \]

\begin{thm}
\label{t.comb}
Consider uniform rotor walk on the comb with any rotor mechanism. Let $n \geq 2$ and $t = \floor{\frac{16}{3} n^3}$.
For any $a>0$ there exist constants $c,C>0$ such that
	\[ \mathbb{P} \{ D_{n - \sqrt{c n log n}} \subset R_t \subset
D_{n + \sqrt{c n log n}} \} > 1 - Cn^{-a}. \]  
\end{thm}

Since the bounding diamonds have area $2n^2(1+o(1))$,
it follows that the size of the range is of order $t^{2/3}$: More precisely, by Borel-Cantelli,
	\[ \frac{\# R_t}{t^{2/3}} \to \left(\frac{3}{2} \right)^{2/3} \]
as $t\to \infty$, almost surely.

\begin{figure}[h!]
\centering
 \begin{tikzpicture}[scale=1.2,>=triangle 60]
 
  \draw[<-,thick,shorten >=4pt] (0,1) -- (1,1);

 \draw[<-,shorten >=4pt] (1,1) -- (2,1);
  \draw[<-,thick,shorten >=4pt] (2,1) -- (3,1);
  \draw[->,thick,shorten >=4pt] (3,1) -- (4,1);
  \draw[<-,thick,shorten >=4pt] (4,1) -- (5,1);

  \draw[->,thick,shorten >=4pt] (-1,1) -- (0,1);
  \draw[<-,thick,shorten >=4pt] (-2,1) -- (-1,1);
  \draw[->,thick,shorten >=4pt] (-3,1) -- (-2,1);
  
         \node[color=black,circle,fill=black,inner sep =0pt,minimum size=4pt] at (-1,1) {};
             \node[color=black,circle,fill=black,inner sep=0pt,minimum size=5pt] at (-2,1) {};
        \node[color=black,circle,fill=black,inner sep=0pt,minimum size=5pt] at (-3,1) {};
         \node[color=black,circle,fill=black,inner sep=0pt,minimum size=5pt] at (0,1) {};
        \node[color=black,circle,fill=black,inner sep=0pt,minimum size=5pt] at (1,1) {};
         \node[color=black,circle,fill=black,inner sep=0pt,minimum size=5pt] at (2,1) {};
          \node[color=black,circle,fill=black,inner sep=0pt,minimum size=5pt] at (3,1) {};
          \node[color=black,circle,fill=black,inner sep=0pt,minimum size=5pt] at (4,1) {};
           \node[color=black,circle,fill=black,inner sep=0pt,minimum size=5pt] at (5,1) {};

\node[color=black]  (o) at (0,0) {o};                 
       \node[color=black,thick] (a1) at (3,0) {$x_1$};
        \node[color=black] (b1) at (-1,-1) {$x_{-1}$};
        \node[color=black] (a2) at (5,-2) {$x_2$};
        \node[color=black] (b2) at (-3,-3) {$x_{-2}$};

 \draw[->,color=black,shorten >=0.3pt] (o) -- (a1);
  \draw[color=black,shorten >=0.33pt] (a1) -- (3,-1);

 \draw[->,color=black,shorten >=0.3pt] (3,-1) -- (b1);
  \draw[color=black] (b1) -- (-1,-2);

 \draw[->,color=black] (-1,-2) -- (a2);
  \draw[color=black] (a2) -- (5,-3);

 \draw[->,color=black] (5,-3) -- (b2);

            \end{tikzpicture}
\caption{An initial rotor configuration on $\mathbb{Z}$ (top) and the corresponding rotor walk.}
\label{f.zigzag}
\end{figure}
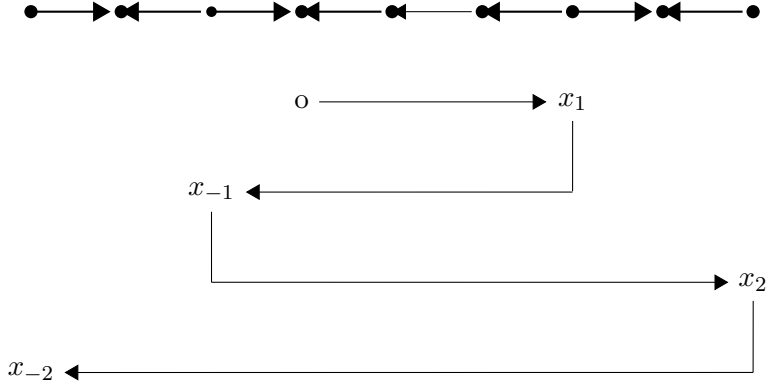

The proof of Theorem~\ref{t.comb} is based on the observation that rotor walk on the comb, viewed at the times when it is on the x-axis, is a rotor walk on $\Z$.
If $0< x_1 < x_2 < \ldots$ are the positions of rotors on the positive x-axis that will send the walker left before right, and $0> x_{-1} > x_{-2} > \ldots$ are the positions on the negative x-axis that will send the walker right before left, then the $x$-coordinate of the rotor walk on the comb follows a zigzag path: right from $0$ to $x_1$, then left to $x_{-1}$, right to $x_2$, left to $x_{-2}$, and so on (Figure~\ref{f.zigzag}).

Likewise, rotor walk on the comb, viewed at the times when it is on the a fixed vertical line $x=k$, is also a rotor walk on $\Z$. Let $0< y_{k,1} < y_{k,2} < \ldots$ be the heights of the rotors on the line $x=k$ above the x-axis that initially send the walker down, and let $0> y_{k,-1} > y_{k,-2} > \ldots$ be the heights of the rotors on the line $x=k$ below the x-axis that initially send the walker up. 

If the initial rotors are i.i.d.\ uniform then the random variables $x_i$ and $y_{k,i}$ 
have mean $2|i|$. Consider the ``bad event'' that one of the random variables $x_i$ or $y_{k,i}$ for $|i| \leq n$ is particularly far from its mean:
 	\[ \mathcal{B}_1 = \mathcal{B}_1(n,c) = \bigcup_{i=-n}^n \{|x_i - 2|i|| > \sqrt{c n\log n}\} \cup \bigcup_{j,k=-n}^{n} \{|y_{k,j} -2|j|| > \sqrt{c n \log{n}}\}. \]
The proof will be completed by the following three lemmas. The first shows that $\mathcal{B}_1$ is unlikely.  The second shows that if $\mathcal{B}_1$ does not occur then the odometer function of the first $m$ excursions is close to the function 
	\begin{equation} \label{e.thediamond} f_m(x,y) = (m-\frac{|x|}{2}-\frac{|y|}{2})^+, \end{equation}
where we write $t^+ := \max(t,0)$. 
Note that the contour lines of this function are diamonds!  
The third lemma says that this forces the range $R_t$ to be close to a diamond.

\begin{lemma}\label{smallprob}
For each $b>0$ there exist $C,c>0$ such that
\[ \mathbb{P}(\mathcal{B}_1(n,c)) < Cn^{-b}. \] 
\end{lemma}

\begin{lemma}\label{shape}
For rotor walk on the comb, $u_m(x,y)$ be the total number of full turns made by the rotor at position $(x,y)$ during the first $m$ excursions. Let $a = \sqrt{c n \log n}$ and
	\[ \mathcal{B}_2 = \mathcal{B}_2(n,c) = \bigcup_{m=1}^n \bigcup_{(x,y) \in \Z^2} \{ f_{m-2a}(x,y) \leq u_m(x,y) \leq f_{m+2a}(x,y) \}^c. \]
Then $\mathcal{B}_2 \subset \mathcal{B}_1$.
\end{lemma}

\begin{lemma}\label{l.theend}
Let $t= \floor{\frac{16}{3}n^3}$ and $a = \sqrt{c n \log n}$ and
\[ \mathcal{B}_3 = \mathcal{B}_3(n,c) = \{D_{n-6a} \subset R_{t} \subset D_{n+6a} \}^c. \] 
Then
	$\mathcal{B}_3(n,c) \subset \mathcal{B}_2(2n,c)$.
\end{lemma}

\begin{proof}[Proof of Lemma~\ref{smallprob}]
Consider the random variable $B_N = \xi_1 + \xi_2 + \ldots + \xi_N$,
where $\xi_i$ is $1$ or $0$ according to whether the rotor at $(i,0)$ will send a particle left before right. 
Note that $\mathbb{P}(x_k \leq N) = \mathbb{P}(B_N \geq k) = \mathbb{P}(S_N \geq 2k-N)$, where $S_N = 2B_N - N$ is a sum of independent zero-mean $\pm 1$-valued random variables.  By the usual Chernoff bound (see, for example, \cite[A.1]{tpm}),
	\[ \mathbb{P}(S_N > a) < e^{-a^2/2N}. \] 
Taking $a = \sqrt{c n \log n}$ and $N = 2k - a$ we obtain
	\[ P(x_k < 2k - a) < e^{-c n \log n / 2N} < e^{-c n \log n / 4n} = n^{-c/4} \]
where we have used that $k \leq n$. Likewise, taking $N=2k+a$ yields
	\[ P(x_k > 2k + a) < P(S_{N}<a) < Cn^{-c/5} \]
for sufficiently large constant $C$. Analogous bounds hold for $k<0$ and for the $y_{k,j}$'s. Now the proof is completed with a union bound
\begin{align*} 
\mathbb{P}(\mathcal{B}_1) 
& \le  \sum_{i=-n}^n \mathbb{P}\left( |x_i - 2i| > a \right) + \sum_{j,k=-n}^n \mathbb{P}\left( |y_{k,j} - 2j| > a \right)  \\
& \le  ((2n) + (2n)^2) (2Cn^{-c/5}) \\
& \le  C' n^{-b}
\end{align*}
where $b = 2-\frac{c}{5}$.
\end{proof}

\begin{proof}[Proof of Lemma~\ref{shape}]
Fix $1 \leq m \leq n$ and a point $(x,y) \in \Z^2$. We must show that on the event $\mathcal{B}_1^c$ we have
	\begin{equation} \label{e.goal} f_{m-2a}(x,y) \leq u_m(x,y) \leq f_{m+2a}(x,y).  \end{equation}
By symmetry we may assume $x,y \geq 0$.
In order to complete $m$ excursions on the comb, the rotor walk viewed on the $x$-axis must complete $m$ zigzags as in Figure~\ref{f.zigzag}, 
	\[ 0 \to x_1 \to x_{-1} \to \ldots \to x_m \to x_{-m} \to 0.  \]
If $x \in (x_K, x_{K+1}]$ then exactly $(m-K)^+$ of these zigzags cross $x$, so
	\[ u_m(x,0) = (m-K)^+. \]
We have used a capital letter to remind you that the index $K$ is random! On the event $\mathcal{B}_1^c \cap \{K < n\}$ we have the inequalities
	\begin{align*} 2K - a &\leq x_K \\ &< x \\ &\leq x_{K+1} \\ &\leq 2(K+1) + a, \end{align*}
which imply $|K- \frac{x}{2}| < a$.  Hence
	\begin{equation} \label{e.xaxis}  (m - \frac{x}{2}-a)^+ \leq u_m(x,0) \leq (m - \frac{x}{2}+a)^+.  \end{equation}
On the event $\mathcal{B}_1^c \cap \{K \geq n\}$ we have $u_m(x,0) = 0$ and $x > 2n - a$, so \eqref{e.xaxis} holds also in this case.

Having taken care of the $x$-axis, we now apply the same argument on each vertical tooth of the comb. The rotor walk viewed on the tooth passing through the point $(x,0)$ completes $M$ zigzags,
	\[ 0 \to y_{x,1} \to y_{x,-1} \to \ldots \to y_{x,M} \to y_{x,-M} \to 0 \]
where $M =u_m(x,0)$.  So on the event $\mathcal{B}_1^c$ we have
	\[ (u_m(x,0) - \frac{y}{2}-a)^+ \leq u_m(x,y) \leq (u_m(x,0) - \frac{y}{2}+a)^+. \]
This bound together with the $x$-axis bound \eqref{e.xaxis} (and the fact that if $u_m(x,0)=0$ then $u_m(x,y)=0$ for all $y \in \Z$) yields \eqref{e.goal}.
Hence $\mathcal{B}_2 \subset \mathcal{B}_1$.
\end{proof}

\begin{proof}[Proof of Lemma~\ref{l.theend}]
For a function $f$ on the vertices of the comb, write
	\[ \| f\| := \sum_{z} \deg(z)f(z) = 2\sum_{x \in \Z} f(x,0) + 2 \sum_{(x,y) \in \Z^2} f(x,y). \]
Summing in diamond layers shows that for the function $f_m$ of \eqref{e.thediamond},
	\[ \| f_m \| = 2 \sum_{\ell=1}^{2m} (4\ell)(m-\frac{\ell}{2}) + O(m^2) = \frac{16}{3}m^3 + O(m^2). \]

Now let $M$ be the (random) number of excursions completed by time $t$. Then
	\[ \| u_M \| \leq t \leq \| u_{M+2} \|.  \]
We first argue that $\{M > \frac32 n\} \subset \mathcal{B}_2(2n)$ for sufficiently large $n$; indeed, on the event $\mathcal{E}:=\{M> \frac32 n\} \cap \mathcal{B}_2(2n)^c$ we have
	\[ t \geq \| u_M \| \geq \| f_{\frac32 n - 2a} \|.  \]
Recall that $t = \frac{16}{3}n^3$. The right side is $(1-o(1))(\frac32)^3 \frac{16}{3}n^3$, so $\mathcal{E}$ is empty for sufficiently large $n$.

Therefore on the event $\mathcal{B}_2(2n)^c$ we have
	\[ \|f_{M-2a} \| \leq \| u_M \| \leq t \leq \| u_{M+2} \| \leq \| f_{M+2+2a} \| \]
and taking cube roots we obtain 
	\[ |M - n| \leq 3a. \]
Finally, writing $A_m = \{u_m >0\}$, on $\mathcal{B}_2(2n)^c$ we have
	 \[ D_{n-5a} \subset A_{n-3a} \subset R_t \subset A_{n+3a+1} \subset D_{n+5a+1} \]
which completes the proof.
\end{proof}

\section{Directed lattices and the mirror model}
\label{s.pinball}

Figure~\ref{f.directed} shows two different orientations of the square grid $\Z^2$: The \emph{F- lattice} has outgoing vertical arrows (N and S) at even sites, and outgoing horizontal arrows (E and W) at odd sites. The \emph{Manhattan lattice} has every even row pointing $E$, every odd row pointing $W$, every even column pointing $S$ and every odd column pointing $N$. In these two lattices every vertex has outdegree $2$, so there is a unique rotor mechanism on each lattice (namely, exits from a given vertex alternate between the two outgoing edges) and a rotor walk is completely specified by its starting point and the initial rotor configuration $\rho$.

In this section we relate the uniform rotor walk on these lattices to percolation and the Lorenz mirror model \cite[\textsection13.3]{perc}.  Consider the \emph{half dual lattice} $\mathbb{L}$, a square grid whose vertices are the points $(x+\frac12, y+\frac12)$ for $x,y \in \Z$ with $x+y$ even.  We consider critical bond percolation on $\mathbb{L}$: each edge of $\mathbb{L}$ is either open or closed, independently with probability $\frac12$.

Note that each vertex $v$ of $\Z^2$ lies on a unique edge $e_v$ of $\mathbb{L}$. We consider two different rules for placing two-sided mirrors at the vertices of $\Z^2$.

\begin{itemize}
\item Manhattan lattice: If $e_v$ is closed then $v$ has a mirror oriented parallel to $e_v$; otherwise $v$ has no mirror.
\item F-lattice: Each vertex $v$ has a mirror, which is oriented parallel to $e_v$ if $e_v$ is closed and perpendicular to $e_v$ if $e_v$ is open.
\end{itemize}

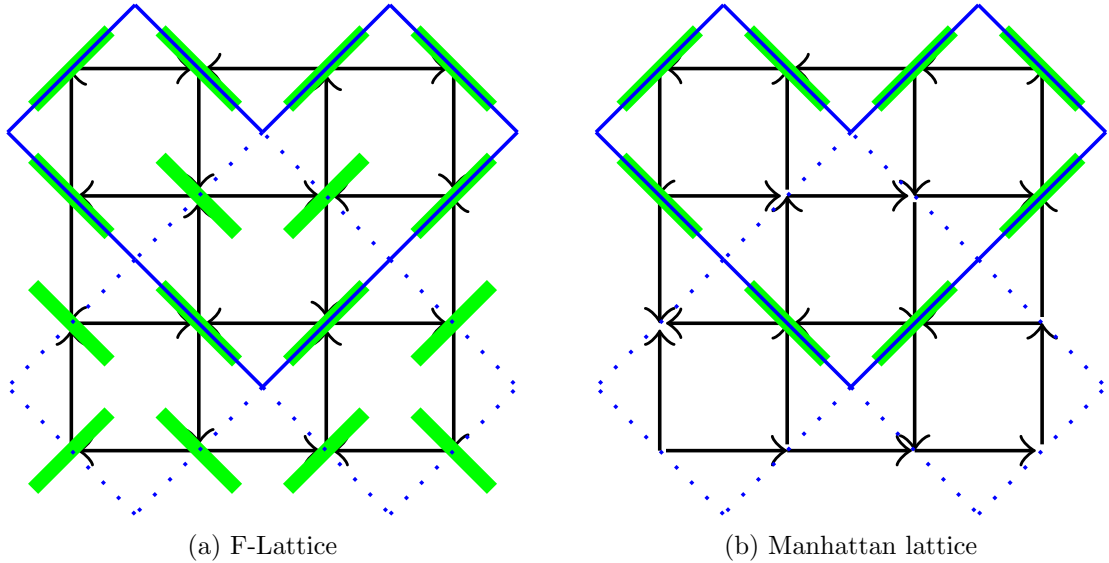
\begin{figure}[h!]
\captionsetup{width=0.8\textwidth}

\centering
    \begin{subfigure}[b]{0.45\textwidth}
        \centering
        \resizebox{\linewidth}{!}{
            \begin{tikzpicture}[scale=0.5]

 \draw[color=black,->, shorten >=0.7pt] (-2,-2) -- (-2,-1);\draw[color=black,->, shorten >=0.7pt] (-1,-2) -- (0,-2);\draw[color=black,->, shorten >=0.7pt] (-1,-2) -- (-2,-2);\draw[color=black, ->,shorten >=0.7pt] (0,-2) -- (0,-1);\draw[color=black, ->, shorten >=0.7pt] (1,-2) -- (0,-2);\draw[color=black, ->, shorten >=0.7pt] (-2,-1) -- (-1,-1);\draw[color=black, ->, shorten >=0.7pt] (-1,-1) -- (-1,0);\draw[color=black, ->, shorten >=0.7pt] (-1,-1) -- (-1,-2);\draw[color=black, ->, shorten >=0.7pt] (0,-1) -- (1,-1);\draw[color=black, ->, shorten >=0.7pt] (0,-1) -- (-1,-1);\draw[color=black, ->, shorten >=0.1pt] (1,-1) -- (1,0);\draw[color=black, ->, shorten >=0.1pt] (1,-1) -- (1,-2);\draw[color=black, ->, shorten >=0.1pt] (-2,0) -- (-2,1);\draw[color=black, ->, shorten >=0.1pt] (-2,0) -- (-2,-1);\draw[color=black, ->, shorten >=0.7pt] (-1,0) -- (0,0);\draw[color=black, ->,shorten >=0.7pt] (-1,0) -- (-2,0);\draw[color=black, ->,shorten >=0.7pt] (0,0) -- (0,1);\draw[color=black, ->,shorten >=0.7pt] (0,0) -- (0,-1);\draw[color=black, ->,shorten >=0.7pt] (1,0) -- (0,0);\draw[color=black, ->,shorten >=0.7pt] (-2,1) -- (-1,1);\draw[color=black, ->,shorten >=0.7pt] (-1,1) -- (-1,0);\draw[color=black, ->,shorten >=0.7pt] (0,1) -- (1,1);\draw[color=black, ->,shorten >=0.7pt] (0,1) -- (-1,1);\draw[color=black, ->,shorten >=0.7pt] (1,1) -- (1,0);

\draw[color=green,ultra thick] (-2.3,-2.3) -- (-1.7,-1.7);
\draw[color=green,ultra thick] (-0.3,-2.3) -- (0.3,-1.7);

\draw[color=green,ultra thick] (-0.7,-1.3) -- (-1.3,-0.7);
\draw[color=green,ultra thick]  (-0.3,-0.3) -- (0.3,0.3);
\draw[color=green,ultra thick]  (-0.7,0.7) -- (-1.3,1.3);

\draw[color=green,ultra thick] (-1.7,-0.3) -- (-2.3,0.3);
\draw[color=green,ultra thick] (-2.3,0.7) -- (-1.7,1.3);
\draw[color=green,ultra thick] (-0.3,0.7) -- (0.3,1.3);
\draw[color=green,ultra thick] (1.3,0.7) -- (0.7,1.3);
\draw[color=green,ultra thick] (0.7,-1.3) -- (1.3,-0.7);
\draw[color=green,ultra thick] (1.3,-2.3) -- (0.7,-1.7);

\draw[color=green,ultra thick] (-0.7,-0.3) -- (-1.3,0.3);
\draw[color=green,ultra thick] (-0.7,-2.3) -- (-1.3,-1.7);
\draw[color=green,ultra thick] (0.7,-0.3) -- (1.3,0.3);
\draw[color=green,ultra thick] (-0.3,-1.3) -- (0.3,-0.7);
\draw[color=green, ultra thick] (-1.7,-1.3) -- (-2.3,-0.7);

\draw[color=blue] (-0.5,-1.5) -- (0.5,-0.5);
\draw[color=blue,dotted]  (0.5,-0.5) -- (-0.5,0.5);
\draw[color=blue,dotted]  (-0.5,0.5) -- (-1.5,-0.5);
\draw[color=blue]  (-1.5,-0.5) -- (-0.5,-1.5);

\draw[color=blue,dotted]  (-1.5,-2.5) --  (-2.5,-1.5);
\draw[color=blue,dotted]   (0.5,-2.5) -- (1.5,-1.5);

\draw[color=blue]   (1.5,0.5) -- (0.5,1.5);
\draw[color=blue]   (-1.5,1.5) -- (-2.5,0.5);

\draw[color=blue,dotted]  (-0.5, -1.5) -- (-1.5,-2.5);

\draw[color=blue]  (0.5, -0.5) -- (1.5,0.5);

\draw[color=blue]  (-0.5, 0.5) -- (0.5,1.5);

\draw[color=blue,dotted]  (-1.5, -0.5) -- (-2.5,-1.5);

\draw[color=blue,dotted]  (-0.5, -1.5) -- (0.5,-2.5);

\draw[color=blue,dotted]  (0.5, -0.5) -- (1.5,-1.5);

\draw[color=blue]  (-0.5, 0.5) -- (-1.5,1.5);

\draw[color=blue]  (-1.5, -0.5) -- (-2.5,0.5);
            \end{tikzpicture}
        }
        \caption{\small F-Lattice}
        \label{fig:subfig8}
    \end{subfigure}
    \begin{subfigure}[b]{0.45\textwidth}
    \centering
        \resizebox{\linewidth}{!}{
            \begin{tikzpicture}[scale=0.5]

            \draw[color=black,->, shorten >=0.3pt] (-1,-2) -- (0,-2);
\draw[color=black, ->, shorten >=0.7pt] (0,-2) -- (1,-2);\draw[color=black, ->, shorten >=0.7pt] (-1,-1) -- (-2,-1);\draw[color=black, ->, shorten >=0.7pt] (0,-1) -- (-1,-1);\draw[color=black, ->, shorten >=0.7pt] (1,-1) -- (0,-1);\draw[color=black, ->, shorten >=0.7pt] (-2,0) -- (-1,0);\draw[color=black, ->, shorten >=0.7pt] (-1,0) -- (0,0);\draw[color=black, ->, shorten >=0.7pt] (0,0) -- (1,0);\draw[color=black,  ->, shorten >=0.7pt] (-1,1) -- (-2,1);\draw[color=black, ->, shorten >=0.7pt] (0,1) -- (-1,1);\draw[color=black, ->, shorten >=0.7pt] (1,1) -- (0,1);
\draw[color=black, ->, shorten >=0.7pt] (-2,-2) -- (-2,-1);

\draw[color=black, <-, shorten >=0.7pt] (-1,-2)  -- (-2,-2);

\draw[color=black, <-, shorten >=0.7pt] (-2,-1) -- (-2,0);
\draw[color=black, <-, shorten >=0.7pt] (-2,0) -- (-2,1);
\draw[color=black, <-, shorten >=0.7pt] (-1,-1) -- (-1,-2);\draw[color=black, <-, shorten >=0.7pt] (-1,0) -- (-1,-1);\draw[color=black, <-, shorten >=0.7pt] (-1,1) -- (-1,0);\draw[color=black, <-, shorten >=0.7pt] (0,-2) -- (0,-1);\draw[color=black, <-, shorten >=0.7pt] (0,-1) -- (0,0);\draw[color=black, <-, shorten >=0.7pt] (0,0) -- (0,1);
\draw[color=black, <-, shorten >=0.7pt] (1,-1) -- (1,-2);\draw[color=black, <-, shorten >=0.7pt] (1,0) -- (1,-1);\draw[color=black, <-, shorten >=0.7pt] (1,1) -- (1,0);

\draw[color=green,ultra thick] (-0.7,-1.3) -- (-1.3,-0.7);
\draw[color=green,ultra thick]  (-0.7,0.7) -- (-1.3,1.3);

\draw[color=green,ultra thick] (-1.7,-0.3) -- (-2.3,0.3);
\draw[color=green,ultra thick] (-2.3,0.7) -- (-1.7,1.3);
\draw[color=green,ultra thick] (-0.3,0.7) -- (0.3,1.3);
\draw[color=green,ultra thick] (1.3,0.7) -- (0.7,1.3);

\draw[color=green,ultra thick] (0.7,-0.3) -- (1.3,0.3);
\draw[color=green,ultra thick] (-0.3,-1.3) -- (0.3,-0.7);

\draw[color=blue] (-0.5,-1.5) -- (0.5,-0.5);

\draw[color=blue,dotted]  (0.5,-0.5) -- (-0.5,0.5);
\draw[color=blue,dotted]  (-0.5,0.5) -- (-1.5,-0.5);

\draw[color=blue]  (-1.5,-0.5) -- (-0.5,-1.5);

\draw[color=blue,dotted]  (-1.5,-2.5) --  (-2.5,-1.5);
\draw[color=blue,dotted]   (0.5,-2.5) -- (1.5,-1.5);

\draw[color=blue]   (1.5,0.5) -- (0.5,1.5);
\draw[color=blue]   (-1.5,1.5) -- (-2.5,0.5);

\draw[color=blue,dotted]  (-0.5, -1.5) -- (-1.5,-2.5);

\draw[color=blue]  (0.5, -0.5) -- (1.5,0.5);

\draw[color=blue]  (-0.5, 0.5) -- (0.5,1.5);

\draw[color=blue,dotted]  (-1.5, -0.5) -- (-2.5,-1.5);

\draw[color=blue,dotted]  (-0.5, -1.5) -- (0.5,-2.5);

\draw[color=blue,dotted]  (0.5, -0.5) -- (1.5,-1.5);

\draw[color=blue]  (-0.5, 0.5) -- (-1.5,1.5);

\draw[color=blue]  (-1.5, -0.5) -- (-2.5,0.5);

            \end{tikzpicture}
        }
        \caption{\small Manhattan lattice}   
        \label{fig:subfig9}
    \end{subfigure}
    \caption{Percolation on $\mathbb{L}$: dotted blue edges are open, solid blue edges are closed. Shown in green are the corresponding mirrors on the $F$-lattice (left) and Manhattan lattice.}
    \label{f.mirrors}
\end{figure}


Consider now the \emph{first glance mirror walk}: Starting at the origin $o$, it travels along a uniform random outgoing edge $\rho(o)$. On its first visit to each vertex $v \neq \Z^2 -\{o\}$, the walker behaves like a light ray: if there is a mirror at $v$ then the walker reflects by a right angle, and if there is no mirror then the walker continues straight. At this point $v$ is assigned the rotor $\rho(v) = (v,w)$ where $w$ is the vertex of $\Z^2$ visited immediately after $v$. On all subsequent visits to $v$, the walker follows the usual rules of rotor walk.

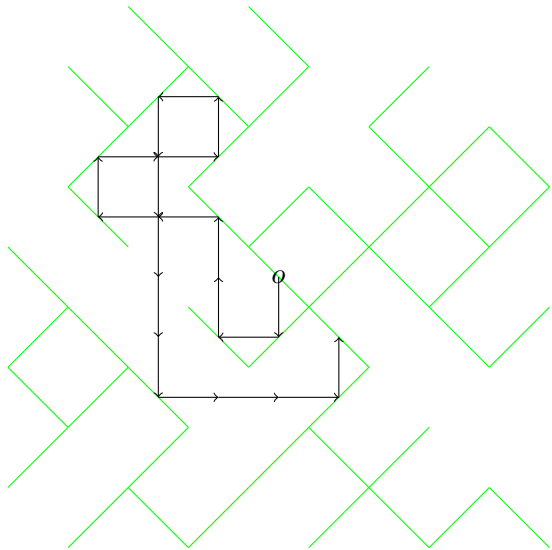
\begin{figure}[h!]
\centering

\begin{tikzpicture}
    [scale=0.8,only marks,
    reddot/.style={fill=red,circle,inner sep=1pt, minimum width=0.5pt},
    bluedot/.style={fill=blue,circle, inner sep=2pt, minimum width=0.5pt},
    blackdot/.style={fill=black,circle, inner sep=2pt, minimum width=0.5pt},
    yellowdot/.style={fill=yellow,circle, inner sep=2pt, minimum width=0.5pt},
    greendot/.style={fill=green,circle, inner sep=2pt, minimum width=0.5pt},    
     spy using outlines={rectangle,lens={scale=3}, size=6cm, connect spies},
]

\draw[color=green] (-3.5,-2.5) -- (-4.5,-3.5);
\draw[color=green] (-4.5,-1.5) -- (-3.5,-2.5);
\draw[color=green] (-3.5,-0.5) -- (-4.5,-1.5);
\draw[color=green] (-4.5,0.5) -- (-3.5,-0.5);
\draw[color=green] (-2.5,-3.5) -- (-3.5,-4.5);
\draw[color=green] (-2.5,-1.5) -- (-3.5,-2.5);
\draw[color=green] (-3.5,-0.5) -- (-2.5,-1.5);
\draw[color=green] (-3.5,1.5) -- (-2.5,0.5);
\draw[color=green] (-2.5,2.5) -- (-3.5,1.5);
\draw[color=green] (-3.5,3.5) -- (-2.5,2.5);
\draw[color=green] (-2.5,-3.5) -- (-1.5,-4.5);
\draw[color=green] (-1.5,-2.5) -- (-2.5,-3.5);
\draw[color=green] (-2.5,-1.5) -- (-1.5,-2.5);
\draw[color=green] (-1.5,3.5) -- (-2.5,2.5);
\draw[color=green] (-2.5,4.5) -- (-1.5,3.5);
\draw[color=green] (-0.5,-3.5) -- (-1.5,-4.5);
\draw[color=green] (-1.5,-0.5) -- (-0.5,-1.5);
\draw[color=green] (-1.5,1.5) -- (-0.5,0.5);
\draw[color=green] (-0.5,2.5) -- (-1.5,1.5);
\draw[color=green] (-1.5,3.5) -- (-0.5,2.5);
\draw[color=green] (0.5,-2.5) -- (-0.5,-3.5);
\draw[color=green] (0.5,-0.5) -- (-0.5,-1.5);
\draw[color=green] (-0.5,0.5) -- (0.5,-0.5);
\draw[color=green] (0.5,1.5) -- (-0.5,0.5);
\draw[color=green] (0.5,3.5) -- (-0.5,2.5);
\draw[color=green] (-0.5,4.5) -- (0.5,3.5);
\draw[color=green] (1.5,-3.5) -- (0.5,-4.5);
\draw[color=green] (0.5,-2.5) -- (1.5,-3.5);
\draw[color=green] (1.5,-1.5) -- (0.5,-2.5);
\draw[color=green] (0.5,-0.5) -- (1.5,-1.5);
\draw[color=green] (1.5,0.5) -- (0.5,-0.5);
\draw[color=green] (0.5,1.5) -- (1.5,0.5);
\draw[color=green] (1.5,-3.5) -- (2.5,-4.5);
\draw[color=green] (2.5,-2.5) -- (1.5,-3.5);
\draw[color=green] (1.5,0.5) -- (2.5,-0.5);
\draw[color=green] (2.5,1.5) -- (1.5,0.5);
\draw[color=green] (1.5,2.5) -- (2.5,1.5);
\draw[color=green] (2.5,3.5) -- (1.5,2.5);
\draw[color=green] (3.5,-3.5) -- (2.5,-4.5);
\draw[color=green] (2.5,-0.5) -- (3.5,-1.5);
\draw[color=green] (3.5,0.5) -- (2.5,-0.5);
\draw[color=green] (2.5,1.5) -- (3.5,0.5);
\draw[color=green] (3.5,2.5) -- (2.5,1.5);
\draw[color=green] (3.5,-3.5) -- (4.5,-4.5);
\draw[color=green] (4.5,-0.5) -- (3.5,-1.5);
\draw[color=green] (4.5,1.5) -- (3.5,0.5);
\draw[color=green] (3.5,2.5) -- (4.5,1.5);

\draw[color=black] (0,0) {} -- (0,-1) {}[->];\draw[color=black] (0,-1) {} -- (-1,-1) {}[->];\draw[color=black] (-1,-1) {} -- (-1,0) {}[->];\draw[color=black] (-1,0) {} -- (-1,1) {}[->];\draw[color=black] (-1,1) {} -- (-2,1) {}[->];\draw[color=black] (-2,1) {} -- (-3,1) {}[->];\draw[color=black] (-3,1) {} -- (-3,2) {}[->];\draw[color=black] (-3,2) {} -- (-2,2) {}[->];\draw[color=black] (-2,2) {} -- (-1,2) {}[->];\draw[color=black] (-1,2) {} -- (-1,3) {}[->];\draw[color=black] (-1,3) {} -- (-2,3) {}[->];\draw[color=black] (-2,3) {} -- (-2,2) {}[->];\draw[color=black] (-2,2) {} -- (-2,1) {}[->];\draw[color=black] (-2,1) {} -- (-2,0) {}[->];\draw[color=black] (-2,0) {} -- (-2,-1) {}[->];\draw[color=black] (-2,-1) {} -- (-2,-2) {}[->];\draw[color=black] (-2,-2) {} -- (-1,-2) {}[->];\draw[color=black] (-1,-2) {} -- (0,-2) {}[->];\draw[color=black] (0,-2) {} -- (1,-2) {}[->];  \draw[color=black] (1,-2) {} -- (1,-1) {}[->]; 
\node at (0,0)[black]{$\small o$};

\end{tikzpicture}
\caption{Mirror walk on the Manhattan lattice.}

\end{figure}

\begin{lemma}
With the mirror assignments described above, uniform rotor walk on the Manhattan lattice or the $F$-lattice has the same law as the first glance mirror walk.
\end{lemma}

\begin{proof}
The mirror placements are such that the first glance mirror walk must follow a directed edge of the corresponding lattice.
The rotor $\rho(v)$ assigned by the first glance mirror walk when it first visits $v$ is uniform on the outgoing edges from $v$; this remains true even if we condition on the past, because all previously assigned rotors are independent of the status of the edge $e_v$ (open or closed), and changing the status of $e_v$ changes $\rho(v)$.
\end{proof}

Write $\beta_e = 1 \{e \text{ is open}\}$. Given the random variables $\beta_e \in \{0,1\}$ indexed by the edges of $\mathbb{L}$, we have described how to set up mirrors and run a rotor walk, using the mirrors to reveal the initial rotors as needed.  The next lemma holds pointwise in $\beta$.

\begin{lemma}
\label{l.cycle}
If there is a cycle of closed edges in $\mathbb{L}$ surrounding $o$, then rotor walk started at $o$ returns to $o$ at least twice before visiting any vertex outside the cycle.
\end{lemma}

\begin{proof}
Denote by $C$ the set of vertices $v$ such that $e_v$ lies on the cycle by $C$, and by $A$ the set of vertices enclosed by the cycle.
Let $w$ be the first vertex not in $A \cup C$ visited by the rotor walk.
Since the cycle surrounds $o$, the walker must arrive at $w$ along an edge $(v,w)$ where $v \in C$.  Since $e_v$ is closed, the walker reflects off the mirror $e_v$ the first time it visits $v$, so only on the second visit to $v$ does it use the outgoing edge $(v,w)$.  Moreover, the two incoming edges to $v$ are on opposite sides of the mirror.  Therefore by minimality of $w$, the walker must use the same incoming edge $(u,v)$ twice before visiting $w$. The first edge to be used twice is incident to the origin by Lemma~\ref{l.excursion1}, so the walk must return to the origin twice before visiting $w$. 
\end{proof}

Now we use a well-known theorem about critical bond percolation: there are infinitely many disjoint cycles of closed edges surrounding the origin. Together with Lemma~\ref{l.cycle} this completes the proof that the uniform rotor walk is recurrent both on the Manhattan lattice and the $F$-lattice.

To make a quantitative statement, consider the probability of finding a closed cycle within a given annulus. The following result is a consequence of the Russo-Seymour-Welsh estimate and FKG inequality.

\begin{thm} 
\cite[11.72]{perc}
\label{t.layer}
Let $S_\ell = [-\ell,\ell]\times [-\ell,\ell]$. Then for all $\ell \geq 1$,
	\[ P(\text{\em there exists a cycle of closed edges surrounding the origin in } S_{3\ell} - S_\ell) > p \]
for a constant $p$ that does not depend on $\ell$.
\end{thm}

Let $u_t(o)$ be the number of visits to $o$ by the first $t$ steps of uniform rotor walk in the Manhattan or $F$-lattice.

\begin{thm}
For any $a>0$ there exists $c>0$ such that
	\[ P(u_t(o) < c \log t) < t^{-a}. \]
\end{thm}

\begin{proof}
By Lemma~\ref{l.cycle}, the event $\{u_t(o)< k \}$ is contained in the event that at most $k/2$ of the annuli $S_{3^j} - S_{3^{j-1}}$ for $j=1,\ldots, \frac{1}{10} \log t$ contain a cyle of closed edges surrounding the origin. Taking $k=c \log t$ for sufficiently small $c$, this event has probability at most 
$t^{-a}$ by Theorem~\ref{t.layer}.
\end{proof}

\begin{figure}[h!]
\captionsetup{width=0.8\textwidth}

\centering
\begin{subfigure}{.5\textwidth}
  \centering
  \includegraphics[width=1\linewidth]{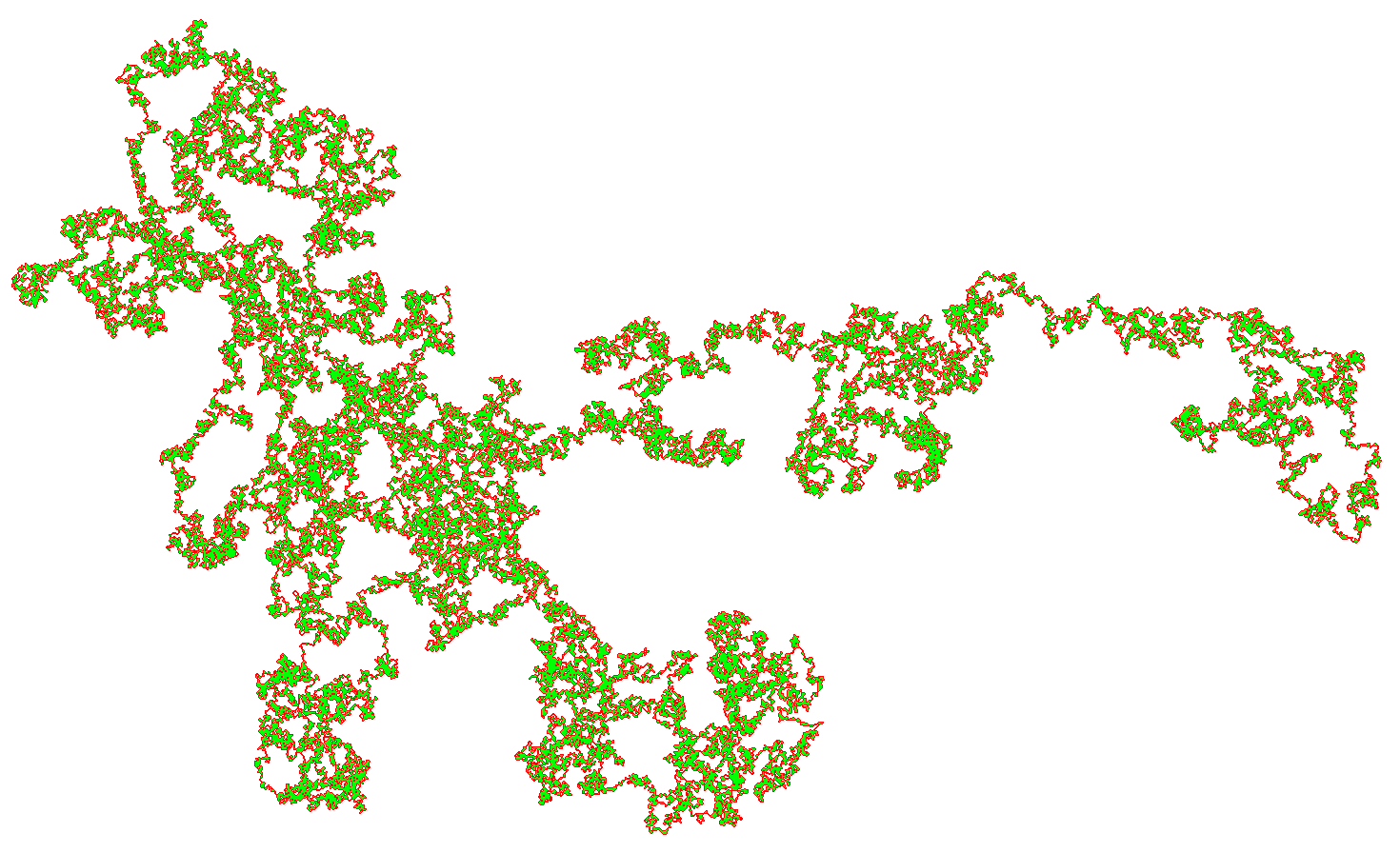}
  \label{fig:sub1}
\end{subfigure}%
\begin{subfigure}{.4\textwidth}
  \centering
  \includegraphics[width=1\linewidth]{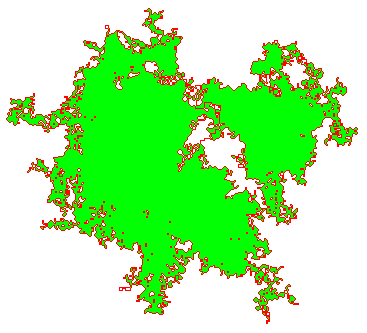}
  \label{fig:sub2}
\end{subfigure}
\caption{Set of sites visited by uniform rotor walk after 250000 steps on the $F$-lattice and the Manhattan lattice (right). Green represents at least two visits to the vertex and red one visit.}
\label{fig:test}
\label{f.manhattan}
\end{figure}

Although we used the same technique to show that the uniform rotor walk on these two lattices is recurrent, experiments suggest that behavior of the two walks is rather different: the number of distinct sites visited in $t$ steps appears to be of order $t^{2/3}$ on the Manhattan lattice but of order $t$ for $F$-lattice. This difference is clearly visible in Figure~\ref{f.manhattan}.

\section{Time for rotor walk to cover a finite Eulerian graph}
\label{s.cover}

Let $(X_t)_{t \geq 0}$ be a rotor walk on a finite connected Eulerian directed graph $G=(V,E)$. The \emph{vertex cover time} is defined by
	\[ t_{vertex} = \min \{t \,:\, \{X_s\}_{s=1}^t = V \}. \]
The \emph{edge cover time} is defined by
	\[ t_{edge} = \min \{t \,:\, \{(X_s,X_{s+1})\}_{s=0}^t = E \}. \]

Yanovski, Wagner and Bruckstein \cite{patrol} show $t_{edge} \leq 2D\#E$ for any Eulerian directed graph. Our next result improves this bound slightly, replacing $2D$ by $D+1$.

\begin{thm}
\label{t.cover}
For rotor walk on a finite Eulerian graph $G$ of diameter $D$, with any rotor mechanism $m$ and any initial rotor configuration $\rho$,
	\[ t_{vertex} \leq D \#E \]
and
	\[ t_{edge} \leq (D+1) \#E. \]
\old{
Moreover, if $G$ is undirected then there exist a rotor mechanism $m$ and initial rotor configuration $\rho$ such that
	\begin{equation} \label{e.lowercover}  t_{vertex}(m,\rho) > \frac12 \floor{\frac{D-1}{2}} \# E. \end{equation}
}
\end{thm}

\begin{proof}
Consider the time $T(n)$ for rotor walk to complete $n$ excursions from $o$.
If $G$ has diameter $D$ then $A_{D} = V$ by Corollary~\ref{c.balls}, and $e_{D+1} = \deg$ by Lemma~\ref{l.excursion3}. It follows that 
	$ t_{vertex} \leq T(D) $
and
	$ t_{edge} \leq T(D+1). $
By Lemma~\ref{l.excursion1}, each edge is used at most once per excursion so $T(n) \leq n \# E$ for all $n \geq 0$.
\old{
For the lower bound we first make the observation that for any undirected graph of diameter $D$ there is a vertex $x_0$ and a connected set of vertices $S_1$ such that at least $\frac12 \# E$ edges of $G$ are incident to $S_1$, and $d(x_0,S_1) \geq \floor{(D+1)/2}$. Namely, let $x_0$ and $x_1$ be vertices of distance $D$, and for $i=0,1$ let 
	\[ S_i = \{v \in V \,:\, d(v,x_i) \leq d(v,x_{1-i}) \}. \]
Then each edge of $G$ is incident either to $S_0$ or $S_1$.
Without loss of generality, at least $\frac12 \#E$ edges are incident to $S_1$. 

Now with this choice of $S_1$ consider the increasing sequence of sets 
	\[ S_{n+1} = S_n \cup \partial S_n  \]
for $n \geq 1$. 
Write $E_n$ for the set of directed edges incident to $S_n$.  
By Theorem~\ref{t.AH} there exists a rotor mechanism $m$ and an initial rotor configuration $\rho$ such that the $n$th excursion of rotor walk started at $x_1$ traverses every directed edge in $E_n$ exactly once, and no other edges.  Since $x_0$ is not adjacent to $S_{\floor{(D-1)/2}}$, the vertex cover time is strictly larger than the total length of the first $\floor{(D-1)/2}$ excursions:
	\[ t_{vertex}(m,\rho) > \sum_{n=1}^{\floor{(D-1)/2}} \# E_n. \]
Since $\# E_n \geq \# E_1 \geq \frac 12 \# E$ the proof is complete.
}
\end{proof}

Bampas et al.\ \cite{lockin} prove a corresponding lower bound: on any finite undirected graph there exist a rotor mechanism $m$ and initial rotor configuration $\rho$ such that $t_{vertex} \geq \frac14 D \#E$. 

\subsection{Hitting times for random walk}
\label{s.hitting}

The upper bounds for $t_{vertex}$ and $t_{edge}$ in Theorem~\ref{t.cover} match (up to a constant factor) those found by Friedrich and Sauerwald \cite{fs} on an impressive variety of graphs: regular trees, stars, tori, hypercubes, complete graphs, lollipops and expanders. Intriguingly, the method of \cite{fs} is different: using a theorem of Holroyd and Propp \cite{holroyd-propp} relating rotor walk to the expected time $H(u,v)$ for \emph{random} walk started at $u$ to hit $v$, they infer that $t_{vertex} \leq K+1$ and $t_{edge} \leq 3K$, where
	\[ K := \max_{u,v \in V} H(u,v) + \frac{1}{2} \left( \# E + \sum_{(i,j) \in E} |H(i,v) - H(j,v) - 1| \right). \]
A curious consequence of the upper bound $t_{vertex} \leq K+1$ of \cite{fs} and the lower bound $\max_{m,\rho} t_{vertex}(m,\rho) \geq \frac14 D \# E$ of \cite{lockin} is the following inequality.

\begin{corollary}
\label{c.hitting}
For any undirected graph $G$ of diameter $D$ we have
	\[ K \geq \frac14 D \#E - 1. \]
\end{corollary}


Is $K$ always within a constant factor of $D \# E$? It turns out the answer is no. To construct a counterexample we will build a graph $G = G_{\ell,N}$ of small diameter which has so few long-range edges that random walk effectively does not feel them  (Figure~\ref{f.thickcycle}). Let $\ell,N \geq 2$ be integers and set $V = \{1,\ldots,\ell\} \times \{1,\ldots,N\}$ with edges $(x,y) \sim (x',y')$ if either $x' \equiv x \pm 1$ (mod $\ell$) or $y'=y$. The diameter of $G$ is $2$: any two vertices $(x,y)$ and $(x',y')$ are linked by the path $(x,y) \sim (x+1,y') \sim (x',y')$.  Each vertex $(x,y)$ has $2N$ short-range edges to $(x\pm 1, y')$ and $\ell-3$ long-range edges to $(x',y)$.  We will argue that if $\ell$ is sufficiently large and $N=\ell^5$, then $K > \frac{1}{10} \ell \# E$, showing that $K$ can exceed $D \# E$ by an arbitrarily large factor.

\begin{figure}
\captionsetup{width=0.8\textwidth}
\centering
\begin{tikzpicture}
         \vertex[color=black,circle,draw,inner sep =1pt,minimum size=4pt] (a1) at (-2,1) {$1$};
         \vertex[color=black,circle,draw,inner sep =1pt,minimum size=4pt] (a2) at (-1,1.5) {$2$};

         \vertex[color=black,circle,draw,inner sep =1pt,minimum size=4pt] (b1) at (0,1) {$1$};
         \vertex[color=black,circle,draw,inner sep =1pt,minimum size=4pt] (b2) at (1,1.5) {$2$};

	      \vertex[color=black,circle,draw,inner sep =1pt,minimum size=4pt] (c1) at (-2,-0.5) {$1$};
         \vertex[color=black,circle,draw,inner sep =1pt,minimum size=4pt] (c2) at (-1,-1) {$2$};
         
               \vertex[color=black,circle,draw,inner sep =1pt,minimum size=4pt] (d1) at (0,-0.5) {$1$};
         \vertex[color=black,circle,draw,inner sep =1pt,minimum size=4pt] (d2) at (1,-1) {$2$};


         
         \draw (a1) -- (b1);
          \draw (a1) -- (c1);
         \draw[dashed] (a1) -- (d1);
                  
          \draw[dashed] (b1) -- (c1);
         \draw (b1) -- (d1);
          \draw (c1) -- (d1);

                \draw (a2) -- (b2);
          \draw (a2) -- (c2);
         \draw[dashed] (a2) -- (d2);
                       
          \draw[dashed] (b2) -- (c2);
         \draw (b2) -- (d2);
          
         \draw (c2) -- (d2);
         
         \draw (a1) -- (b2); 
         \draw (a2) -- (b1);
         
         \draw (a1) -- (c2);
         \draw (a2) -- (c1);

          \draw (b1) -- (d2);
          \draw (b2) -- (d1);
          
          \draw (d1) -- (c2);
          \draw (d2) -- (c1);
          
\end{tikzpicture}
\caption{The thick cycle $G_{\ell,N}$ with $\ell=4$ and $N=2$. Long range edges are dotted and short range edges are solid.
\label{f.thickcycle}}
\end{figure}
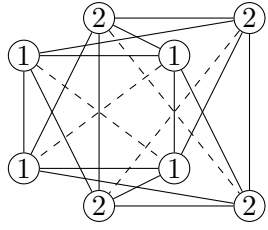

\old{

The main point will be to lower bound the hitting time difference $|H(i,v)-H(j,v)+1|$ for neighboring sites $i,j \in V$. As a preliminary step, note that for any $i,j\in V$, not necessarily neighbors, we have the upper bound
	\[ \EE |\tau(i,v)-\tau(j)| < \ell^2 \]
achieved by the following coupling: if we use a short-range edge at time $t$ then we take $Y^i_t = Y^j_t$ and we let $X^i_t, X^j_t$ perform independent lazy random walk steps on the $\ell$-cycle.  If we use a long-range edge at time $t$ then we can take $(X^i_t, Y^i_t) = (X^j_t, Y^j_t)$.  Then $|H(i,v) - H(j,v)|$ is at most the expected time for two independent lazy random walks to meet on the $\ell$-cycle, which is at most $\ell^2$.
}

Write $Z^i_t = (X^i_t,Y^i_t)$ for random walk on $G$ started at $i$. 
We couple the walks $(Z^i_t)_{t \geq 0}$ and $(Z^j_t)_{t \geq 0}$ as follows. At each time $t$ either both walks will use a short-range edge or both walks will use a long-range edge.
If they use a short-range edge, then we move the $X$ coordinates independently and take $Y^i_t = Y^j_t$.  If they use a long-range edge, then we take $X^i_t = X^j_t$.

Let $T_{short}$ be the first time a short-range edge is used, and consider the hitting times
 	\[ \tau(i) = \min \{t \geq 1 \,:\, (X^i_t, Y^i_t)=(1,1)\} \]
	\[ \sigma(i) = \min \{t \geq T_{short} \,:\,  X^i_t = 1\}. \]
Now fix starting vertices $i$ and $j$ with $i_2, j_2 >1$. Then neither walk can hit $v=(1,1)$ before time $T_{short}$, so $\sigma(i) \leq \tau(i)$. Decompose the hitting time $\tau(i)$ into two pieces
	\[ \tau(i) = \sigma(i) + \sigma_2(i). \]
Both $Y^i_{\sigma(i)}$ and $Y^j_{\sigma(j)}$ are uniformly distributed on $\{1,\ldots,L\}$, so $\EE \sigma_2(i) = \EE \sigma_2(j)$. Hence
	\begin{equation} \label{e.hdiff} H(i,v) - H(j,v) 
	= \EE \sigma(i) - \EE \sigma(j). \end{equation}
To estimate the right side, we couple the random walks on $G$ to random walks $\til X^i_t$, $\til X^j_t$ on the $\ell$-cycle as follows.
Let $T_{long}$ be the first time a long-range edge is used. For $k \in \{i,j\}$ let 
$\til X^k_t = X^k_t$ for $0 \leq t < T_{long}$, and let the increments $\til X^k_t - \til X^k_{t-1}$ be independent of $X$ for $t \geq T_{long}$. Let
	\[ \til \sigma (k) = \min \{t \geq 1 \,:\, \til X^k_t = 1\}. \]
On the event
\[ \mathcal{A} := \{ \max(\sigma(i), \sigma(j)) < T_{long} \} \]
we have $T_{short}=1$ and $\til \sigma(k) = \sigma(k)$ for $k \in \{i,j\}$. Hence
	\begin{equation} \label{e.comparison} \EE | \sigma(i) - \sigma(j) - (\til \sigma(i) - \til \sigma(j)) | \leq \PP(\mathcal{A}^c) \left(\max_{i',j' \in V} \EE |\sigma(i')-\sigma(j')| + \EE (\til \sigma(i) + \til \sigma(j)) \right). \end{equation}

Since the probability of using a long-range edge at each fixed time $t$ is $(\ell-3)/(2N+\ell-3)$, we have
	\[ \PP(\mathcal{A}^c) \leq \EE \sum_{t=0}^\infty 1 \{t \leq \max(\sigma(i), \sigma(j)) \} \frac{\ell-3}{2N+\ell-3} < \EE (\sigma(i) + \sigma(j)) \frac{\ell}{2N}. \]
Now we use the explicit formula for random walk on the $\ell$-cycle, $\EE \til \sigma(i) = (i_1-1) (\ell-i_1+1)$. In particular, we have 	
	\[ |\EE \til \sigma(i) - \EE \til \sigma(j))| \geq \frac{\ell}{3} \qquad \forall (i,j) \in E_1 \]
where
	\[ E_1 := \left \{ (i,j) \in E \,: |i_i - j_1| = 1, \, \, i_2>1, \, j_2>1, \, \left| i_1 - \frac{\ell}{2} \right| > \frac{\ell}{5} \right \}. \]
Now take $N= \ell^5$ and $\ell$ sufficiently large. Since $\max_{i' \in V} \EE \sigma(i') < \ell^2/2$, the right side of \eqref{e.comparison} is $<1$. Therefore by \eqref{e.hdiff} we have
	 \[ K > \sum_{(i,j) \in E_1} |H(i,v) - H(j,v)+1| > (\frac12 \# E) (\frac{\ell}{3} - 2) > \frac{\ell}{10} \# E \]
as desired.

Note that Corollary~\ref{c.hitting} is a fact purely about random walk on a graph. Can it be proved without resorting to rotor walk?

\section*{Acknowledgements}
This work was initiated while the first two authors were visiting Microsoft Research in Redmond, WA. We thank Sam Watson for help with some of the simulations, and Tobias Friedrich for bringing to our attention references \cite{lockin} and \cite{patrol}.

\bibliographystyle{plain}
  \bibliography{range}

\end{document}